\documentclass[11pt,reqno,a4paper]{amsart}

\usepackage{graphicx}
\usepackage{color, xcolor} 

\usepackage[colorlinks,pdfpagelabels,pdfstartview=FitH,bookmarksopen=true,bookmarksnumbered=true,linkcolor=blue,plainpages=false,hypertexnames=false,citecolor=red]{hyperref}
\usepackage{mathrsfs}
\usepackage{verbatim, amssymb, amsmath, amsfonts, amsthm}
\usepackage{latexsym}

\allowdisplaybreaks %spezza gli eqnarray lunghi se finisce la pagina

\newtheorem{theorem}{Theorem}[section]

\newtheorem{lemma}[theorem]{Lemma}



\newcommand{\fr}{\frac }

\newcommand{\R}{\mathbb{R}}

\def\sideremark#1{\ifvmode\leavevmode\fi\vadjust{\vbox to0pt{\vss% the remark
			\hbox to 0pt{\hskip\hsize\hskip1em%                          will appear only
				\vbox{\hsize1.1cm\tiny\raggedright\pretolerance10000%          on the side
					\noindent #1\hfill}\hss}\vbox to5pt{\vfil}\vss}}}%

\begin{document}

\title[Asymptotic analysis and quantization]{Asymptotic analysis and energy quantization\\ for  the Lane-Emden problem in dimension two}

\author[]{F. De Marchis,  M. Grossi, I. Ianni, F. Pacella}

\address{F. De Marchis,  M. Grossi,  F. Pacella, Dipartimento di Matematica, Universit\`a degli Studi \emph{Sapienza}, P.le Aldo Moro 5, 00185 Roma, Italy}
\address{Isabella Ianni, Dipartimento di Matematica e Fisica, Universit\`a degli Studi della Campania \emph{Luigi Vanvitelli}, V.le Lincoln 5, 81100 Caserta, Italy}

\thanks{2010 \textit{Mathematics Subject classification:}  }


\thanks{Research partially supported by: PRIN $201274$FYK7$\_005$ grant and INDAM - GNAMPA}

\begin{abstract}
We complete the study  of the asymptotic behavior, as $p\rightarrow +\infty$, of the positive solutions to 
\[
\left\{\begin{array}{lr}-\Delta u= u^p &  \mbox{ in }\Omega\\
u=0 &\mbox{ on }\partial \Omega
\end{array}\right.
\]
 when $\Omega$ is any smooth bounded  domain in $\mathbb R^2$, started in \cite{DIPpositive}. In particular we  show  quantization of the energy to multiples of $8\pi e$  and prove convergence to $\sqrt{e}$ of the $L^{\infty}$-norm,  thus confirming  the conjecture made in \cite{DIPpositive}.
 
\end{abstract}

\maketitle

\section{Introduction}
This paper focuses on the asymptotic analysis, as $p\rightarrow +\infty$, of families of solutions to  the Lane-Emden  problem 
\begin{equation}\label{problem}
\left\{\begin{array}{lr}-\Delta u= u^p &  \mbox{ in }\Omega\\
u=0 &\mbox{ on }\partial \Omega
\\
u>0 & \mbox{ in }\Omega
\end{array}\right. \tag{$\mathcal P_p$}
\end{equation}
where  $\Omega$ is any smooth bounded planar domain.

\

This line of investigation  started in \cite{RenWeiTAMS1994, RenWeiPAMS1996} for families $u_p$ of least energy solutions, for which a one-point concentration behavior in the interior of $\Omega$ is proved, as well as the \emph{$L^{\infty}$-bounds}  
\begin{equation}\label{aha}
\sqrt{e}\leq\lim_{p\rightarrow +\infty} \|u_p\|_{\infty}\leq  C \end{equation}
and the following estimate
\[\lim_{p\rightarrow +\infty} p\|\nabla u_p\|_2^2= 8\pi e.\]
The bound in \eqref{aha} was later improved in  \cite{AdiGrossi}, where it was shown that for families of least energy solutions  the following limit holds true:
\begin{equation}\label{LimiteNormaLInfinito}\lim_{p\rightarrow +\infty}\|u_p\|_{\infty}=\sqrt{e}. \end{equation}
Moreover in \cite{AdiGrossi} and \cite{ElMGrossi} the  Liouville equation in the whole plane
 \begin{equation}\label{LiouvilleEquationINTRO}
 \left\{
 \begin{array}{lr}
 -\Delta U=e^U\quad\mbox{ in }\R^2\\
 \int_{\R^2}e^Udx= 8\pi.
 \end{array}
 \right.
 \end{equation}
was identified to be a \emph{limit problem} for the Lane-Emden equation. Indeed in \cite{AdiGrossi}  it was proved that suitable rescalings around the maximum point of any least energy solution to \eqref{problem} converge, in $C^2_{loc}(\mathbb{R}^2)$, to the regular solution \begin{equation}\label{definizioneU}
U(x)=\log\left(\fr1{1+\fr18 |x|^2}\right)^2
\end{equation} of \eqref{LiouvilleEquationINTRO}. 
Hence least energy solutions exhibit only one concentration point and the local limit profile is given by \eqref{definizioneU}. More general solutions having only one peak have been recently studied in \cite{progress}, where their Morse index is computed and connections with the question of the uniqueness of positive solutions in convex domains are shown.

\

Observe that  when $\Omega$ is a ball any  solution  to \eqref{problem}  is radial by Gidas, Ni and Nirenberg result (\cite{GNN})  and so the least energy  is the unique solution for any $p>1$. 

\

In general in non-convex domains there may be families of solutions to \eqref{problem} other than the least energy ones.
This is the case, for instance,  of those found in \cite{EspositoMussoPistoiaPos}  when the domain $\Omega$ is not simply connected, which have higher energy, precisely 
\[\lim_{p\rightarrow +\infty} p\|\nabla u_p\|_2^2= 8\pi e\cdot k,\]
for any fixed integer $k\geq 1$. These solutions exhibit a concentration phenomenon  at $k$ distinct points in $\Omega$ as  $p\rightarrow +\infty$ and  their $L^{\infty}$-norm satisfies the same limit as in \eqref{LimiteNormaLInfinito}.

 \

The question of characterizing the behavior of any family $u_p$ of solutions to \eqref{problem} naturally arises. An almost complete answer has been recently given in  \cite{DIPpositive}  in any  general smooth bounded domain $\Omega$,
under the uniform energy bound assumption
\begin{equation}
\label{energylimit}
 p\|\nabla u_p\|_2^2\leq C
\end{equation}
(see also \cite{DeMarchisIanniPacellaJEMS}, where this general asymptotic analysis was started and the related papers \cite{DIPAMPA,DIPSurvey}).
The results in  \cite{DIPpositive} show that under the assumption \eqref{energylimit} the solutions to \eqref{problem} are necessarily  \emph{spike-like} and that \emph{the energy is quantized}. More precisely in  \cite[Theorem 1.1]{DIPpositive} it is proved that, up to a subsequence, there exists an integer $k\geq 1$ and $k$ distinct points $x_i\in\Omega$, $i=1,\ldots, k$, such that, setting
\[\mathcal S=\{x_1,\ldots, x_k\},\]
one has
\begin{equation}\label{pu_va_a_zerovecchia}
\lim_{p\rightarrow +\infty}\sqrt{p}u_{p}=0\ \mbox{ in $C^2_{loc}(\bar\Omega\setminus\mathcal{S})$}
\end{equation}
and the energy satisfies
\begin{equation}\label{energylimitm_i^2vecchia}
\lim_{p\rightarrow +\infty}p\|\nabla u_{p}\|_2^2= 8\pi \sum_{i=1}^k m_i^2,
\end{equation} 
where  $m_i$'s are positive constants given by
\begin{equation}
m_i=\lim_{\delta\rightarrow 0}\lim_{p\rightarrow +\infty}\max_{\overline{B_{\delta}(x_i)}} u_{p}\end{equation}
which satisfy \begin{equation}\label{mistimati}m_i\geq \sqrt{e}.
\end{equation}
Furthermore  the location of the concentration points is shown to depend on the Dirichlet Green function $G$  of $-\Delta$ in $\Omega$  and on its regular part $H$
\begin{equation}\label{HparteregolareGreen}
H(x,y)=G(x,y)+\frac{\log(|x-y|)}{2\pi}
\end{equation}
according to the following system
\[
m_i \nabla_x H(x_i,x_i)+\sum_{\ell\neq i } m_\ell \nabla_x G(x_i,x_\ell)=0,
\]
and moreover
\[
\lim_{p\rightarrow +\infty}pu_{p}=8\pi \sum_{i=1}^k m_i G(\cdot,x_i)\mbox{ \ in $C^2_{loc}(\bar\Omega\setminus \mathcal S)$}.
\]
In \cite[Lemma 4.1]{DIPpositive} it is also proved that a suitable rescaling of $u_p$ around each concentration point, in the spirit of the one done in \cite{AdiGrossi} for the least energy solutions,  converges to the regular solution $U$ in \eqref{definizioneU}.
\\
\\
Observe that \eqref{pu_va_a_zerovecchia} and \eqref{mistimati} immediately imply the following bound on the $L^{\infty}$-norm:
	\begin{equation}\label{daMigliorare}\sqrt{e}\leq\lim_{p\rightarrow +\infty} \|u_p\|_{\infty}\leq  C.\end{equation}
In \cite{DIPpositive} it was  conjectured that for all  solutions to \eqref{problem}, under the assumption  \eqref{energylimit}, one should have the equality in \eqref{mistimati}.\\
\\
Here we complete the analysis in \cite{DIPpositive} proving this conjecture, namely we show the following:
\begin{theorem}\label{teo:convmi}
	\[m_i=\sqrt{e}, \quad \forall i=1,\ldots, k.\]
\end{theorem}
This result implies, by \eqref{pu_va_a_zerovecchia} and \eqref{mistimati},  a sharp improvement of \eqref{daMigliorare}:
\begin{theorem}[$L^{\infty}$-norm limit]
	\label{teo:ConvNorma}
	Let  $u_p$ be a family of solutions to \eqref{problem}  and assume that \eqref{energylimit} holds. Then
	\[\lim_{p\rightarrow +\infty}\|u_p\|_{\infty}=\sqrt{e}.\]
\end{theorem}
On the other side, by  \eqref{energylimitm_i^2vecchia}, Theorem \ref{teo:convmi}
implies  a quantization of the energy to integer multiples of $8\pi e$ as $p$ goes to infinity. Our final asymptotic results can be summarized as follows:
\begin{theorem}[Complete asymptotic behavior \& quantization]
	\label{teoAsymptotic} 
	Let $u_p$ be a family of solutions to \eqref{problem}  and assume that \eqref{energylimit} holds.
	Then there exist a finite number $k$   of distinct  points $x_i\in\Omega$, $i=1,\ldots, k$   and a sequence $p_n\rightarrow +\infty$ as $n\rightarrow +\infty$
	such that setting \[\mathcal S:=\{x_1,\ldots, x_k\}\]
	one has
	\begin{equation}\label{pu_va_a_zeroTeo}
	\lim_{n\rightarrow \infty}\sqrt{p_n}u_{p_n}= 0\ \mbox{ in $C^2_{loc}(\bar\Omega\setminus\mathcal{S})$.}
	\end{equation}
	The concentration points $x_i, \ i=1,\ldots, k$ satisfy  the system
	\begin{equation}\label{x_j relazioneTeo}
	\nabla_x H(x_i,x_i)+\sum_{i\neq \ell}\nabla_x G(x_i,x_\ell)=0.
	\end{equation}
	Moreover 
	\begin{equation}\label{convergenzapup}
	\lim_{n\rightarrow \infty}p_nu_{p_n}(x)=  8\pi \sqrt{e} \sum_{i=1}^k G(x,x_i)\  \mbox{ in } 
	C^2_{loc}(\bar\Omega\setminus\mathcal S)  
	\end{equation}
	and the energy satisfies
	\begin{equation}
	\label{convergenzaenergia}
	\lim_{n\rightarrow \infty}p_n\int_\Omega |\nabla u_{p_n}(x)|^2\,dx= 8\pi e\cdot k.
	\end{equation}
\end{theorem}

\

\section{Proof of Theorem \ref{teo:convmi}}\label{section:quantization}

Let $k\geq 1$ and $x_i\in\Omega$, $i=1,\ldots,k$ be as in the introduction and let us keep the notation $u_p$ to denote the corresponding subsequence of the family $u_p$ for which the results in \cite{DIPpositive} hold true.\\  

In particular (see \cite[Theorem 1.1 \& Lemma 4.1]{DIPpositive}) for $r>0$ such that $B_{3r}(x_j)\subset\Omega,$ $\forall  j=1,\ldots, k$ and $B_{3r}(x_j)\cap B_{3r}(x_i)=\emptyset$, $\forall j=1,\ldots, k$, $j\neq i$, letting $y_{i,p}\in\Omega$ be the sequence defined as
 \begin{equation}
\label{ySonoMax}
u_{p}(y_{i,p}):=\max_{\overline{B_{2r}(x_i)}} u_{p}
\end{equation}
it follows that
\begin{equation}
\label{convMax}
\lim_{p\rightarrow +\infty}y_{i,p}= x_i,
\end{equation}
\begin{equation}
\label{convValMax} \lim_{p\rightarrow +\infty}u_{p}(y_{i,p})= m_i,
\end{equation}
\begin{equation} 
\label{defEpsilon}
\lim_{p\rightarrow +\infty}\varepsilon_{i,p}\left(:=\left[ p u_{p}(y_{i,p})^{p-1}\right]^{-1/2}\right)= 0
\end{equation}
 and setting
\begin{equation}
\label{defRiscalataMaxvecchia}
w_{i,p}(y):=\fr{p}{u_{p}(y_{i,p})}(u_{p}(y_{i,p}+\varepsilon_{i,p} y)-u_{p}(y_{i,p})),\quad  y\in  \Omega_{i,p}:=\frac{\Omega-y_{i,p}}{\varepsilon_{i,p}},
%
%   B_{\frac{r}{\varepsilon_{i,p}}}\left(\frac{x_i-y_{i,p}}{\varepsilon_{i,p}} \right),
\end{equation}
then
\begin{equation}\label{convRiscalateNeiMaxVecchia}
\lim_{p\rightarrow +\infty}w_{i,p}= U\ \mbox{ in }\ C^2_{loc}(\R^2),
\end{equation}
where $U$ is as in \eqref{definizioneU}.
\\
\\
Furthermore by the result in \cite[Proposition 4.3 \& Lemma 4.4]{DIPpositive} 
we have that for any $\gamma\in (0,4)$ there exists $R_{\gamma}>1$ such that
\begin{equation}\label{stimaIacopetti}
w_{i,p}(z)\leq \left(4-\gamma\right)\log\frac{1}{|z|}+\widetilde C_{\gamma},\qquad\forall i=1,\ldots,k
\end{equation}
for some $\widetilde C_{\gamma}>0$, provided $R_\gamma\leq |z|\leq \frac{r}{\varepsilon_{i,p}}$ and $p$ is sufficiently large.\\
\\
The pointwise estimate \eqref{stimaIacopetti} implies the following uniform bound, which will be the key to use the dominated convergence theorem in the proof of Theorem \ref{teo:convmi}:
\begin{lemma}
\begin{equation}\label{arietta2} 
0\leq \left(1+\frac{w_{j,p}(z)}{p}\right)^{p}\leq 
\left\{   
\begin{array}{lr}
 1 &\mbox{ for }|z|\leq R_{\gamma}
\\
 C_{\gamma}\frac{1}{|z|^{4-\gamma}} & \mbox{ for  
	$R_\gamma\leq 
	|z|\leq \frac{r}{\varepsilon_{j,p}}$}
\end{array}\right..
\end{equation}		
\end{lemma}

\begin{proof}
Observe that by \eqref{convMax}
	\[
	B_{r}(y_{i,p})\subset B_{2r}(x_i),\mbox{ for $p$ sufficiently large},
	\]
	as a consequence
	\begin{equation}
	\label{wNegativa}
	w_{i,p}\leq 0, \ \mbox{ in }B_{\frac{r}{\epsilon_{i,p}}}(0)\ (\subset\Omega_{i,p}), \mbox{ for $p$ large},
	\end{equation}	
which implies the first bound in \eqref{arietta2}.
	\\
	For $p$ sufficiently large, by \eqref{wNegativa} and \eqref{stimaIacopetti}, we also get the second bound in \eqref{arietta2}:
	\[
	0\leq  \left( 1+\frac{w_{j,p}(z)}{p}\right)^{p}=
	e^{p\log \left( 1+\frac{w_{j,p}(z)}{p} \right)}
	\leq   e^{w_{j,p}(z)}\leq C_{\gamma}\frac{1}{|z|^{4-\gamma}}
	\]
	for  
	$R_\gamma\leq 
	|z|\leq \frac{r}{\varepsilon_{j,p}}$.
\end{proof}

\

\begin{proof}[Proof of Theorem \ref{teo:convmi}]
	Observe that by the assumption \eqref{energylimit} and H\"older inequality
	\begin{eqnarray*}
		(0\leq)\ p\int_{\Omega}u_p^p(x) dx & \leq &  p^{\frac{1}{p+1}} |\Omega|^{\frac{1}{p+1}}\left[p\int_{\Omega}|\nabla u_{p}|^2 dx\right]^{\frac{p}{p+1}}
		\\
		&= & p\int_{\Omega}|\nabla u_{p}|^2 dx +o_p(1)
		\\
		&\overset{\eqref{energylimit} }{\leq }& C +o_p(1),
	\end{eqnarray*}
	so that, by the properties  of the Green function $G$,
	\begin{eqnarray}
	\label{stimaGreenFormula}
	\int_{\Omega\setminus B_{2r}(x_j)} G(y_{j,p},x) u_p^p(x) dx
	& \leq &
	C_{r}\int_{\Omega\setminus B_{2r}(x_j)}  u_p^p(x) dx
	\nonumber
	\\
	& \leq & 
	C_{r}\int_{\Omega}  u_p^p(x) dx =O\left(\frac1p\right)
	\end{eqnarray}
	and similarly, observing that for $p$ large enough the points $y_{j,p}\in B_{\frac{r}{2}}(x_j)$  by \eqref{convMax} and $B_{\frac{r}{2}}(x_j)\subset B_{r}(y_{j,p})\subset B_{2r}(x_j)$, also
	\begin{eqnarray}
	\label{cambioPalletta}
	\int_{B_{2r}(x_j)\setminus   B_{r}(y_{j,p})  }
	G(y_{j,p},x) u_p^p(x) dx &\leq & \int_{\{\frac{r}{2}<|x-x_j|<2r\}}
	G(y_{j,p},x) u_p^p(x)dx
	\nonumber
	\\
	&\leq & C_{\frac{r}{2}}\int_{\Omega} u_p^p(x) dx =O\left(\frac1p\right).
	\end{eqnarray}
	By the Green representation formula, using the previous estimates, we then get
	\begin{eqnarray}\label{GreenRepresPart}
	u_{p}(y_{j,p})&=&\int_{\Omega} G(y_{j,p},x) u_p^p(x) dx
	\nonumber\\
	&=& \int_{B_{2r}(x_j)} G(y_{j,p},x) u_p^p(x) dx+\int_{\Omega\setminus B_{2r}(x_j)} G(y_{j,p},x) u_p^p(x) dx
	\nonumber\\
	&\overset{{\scriptsize{\begin{array}{cc}\eqref{stimaGreenFormula}\\\eqref{cambioPalletta}\end{array}}}}{=} & \int_{B_{r}(y_{j,p})} G(y_{j,p},x) u_p^p(x) dx+o_p(1)
	\nonumber\\
	&\overset{\eqref{defRiscalataMaxvecchia}}{=}& \int_{B_{\frac{r}{\varepsilon_{j,p}}}(0)} G(y_{j,p},y_{j,p}+\varepsilon_{j,p}z) \left(1+\frac{w_{j,p}(z)}{p}\right)^{p} dz+o_p(1)
	\nonumber\\
	&\overset{\eqref{HparteregolareGreen}}{=} &
	\frac{u_{p}(y_{j,p})}{p}\int_{B_{\frac{r}{\varepsilon_{j,p}}}(0)} H(y_{j,p},y_{j,p}+\varepsilon_{j,p}z) \left(1+\frac{w_{j,p}(z)}{p}\right)^{p} dz
	\nonumber
	\\
	&&
	-\frac{u_{p}(y_{j,p})}{2\pi p}\int_{B_{\frac{r}{\varepsilon_{j,p}}}(0)} \log|z| \left(1+\frac{w_{j,p}(z)}{p}\right)^{p} dz  
	\nonumber
	\\
	&&
	-\frac{u_{p}(y_{j,p})\log\varepsilon_{j,p} }{2\pi p}\int_{B_{\frac{r}{\varepsilon_{j,p}}}(0)} \left(1+\frac{w_{j,p}(z)}{p}\right)^{p} dz +o_p(1)
	\nonumber
	\\
	&=& A_p +B_p+ C_p +o_p(1).
	\end{eqnarray}
	Since  $H$ is smooth and  $x_j\not\in\partial\Omega$,
	 by \eqref{convMax} and \eqref{defEpsilon} we get 
	\[\lim_{p\rightarrow +\infty} H(y_{j,p},y_{j,p}+\varepsilon_{j,p}z)= H(x_j,x_j), \ \mbox{for any $z$}\in\mathbb R^2,\]
	so by  \eqref{convValMax}, the convergence \eqref{convRiscalateNeiMaxVecchia} and the uniform bounds in \eqref{arietta2} we  can  apply the dominated convergence theorem,  and since the function 
	$z\mapsto\frac{1}{|z|^{4-\gamma}}$ is integrable in
	$\{|z|> R_{\gamma}\}$ choosing $\gamma\in (0,2)$  we deduce
	\begin{eqnarray*}  
		&& \lim_{p\rightarrow +\infty} u_{p}(y_{j,p})\int_{B_{\frac{r}{\varepsilon_{j,p}}}(0)} H(y_{j,p},y_{j,p}+\varepsilon_{j,p}z) \left(1+\frac{w_{j,p}(z)}{p}\right)^{p} dz\\
		&&\qquad\qquad\qquad\qquad\qquad\qquad\qquad =m_j H(x_j,x_j)\int_{\R^2}e^U\overset{\eqref{LiouvilleEquationINTRO}}{=} 8\pi\, m_j H(x_j,x_j),
	\end{eqnarray*}
	from which 
	\begin{equation}\label{termineA_p}
	A_p := \frac{u_{p}(y_{j,p})}{p}\int_{B_{\frac{r}{\varepsilon_{j,p}}}(0)} H(y_{j,p},y_{j,p}+\varepsilon_{j,p}z) \left(1+\frac{w_{j,p}(z)}{p}\right)^{p} dz=o_p(1).
	\end{equation}
	For the second term in  \eqref{GreenRepresPart} we apply again the dominated convergence theorem, 
	using  \eqref{arietta2} and observing now that
	the function 
	$z\mapsto \frac{\log |z|}{|z|^{4-\gamma}}$ is integrable 
	in $\{|z|>R_{\gamma}\}$ and that $z\mapsto\log |z|$ is integrable in
	$\{|z|\leq R_{\gamma}\}$. 
	Hence we get
	\[\lim_{p\rightarrow +\infty} u_{p}(y_{j,p})\int_{B_{\frac{r}{\varepsilon_{j,p}}}(0)} \log |z|\left(1+\frac{w_{j,p}(z)}{p}\right)^{p} dz= m_j \int_{\mathbb R^2}\log |z|e^{U(z)}dz < +\infty\]
	and this implies that
	\begin{eqnarray}
	\label{termineB_p}
	B_p &:=&-\frac{u_{p}(y_{j,p})}{2\pi p}\int_{B_{\frac{r}{\varepsilon_{j,p}}}(0)} \log|z| \left(1+\frac{w_{j,p}(z)}{p}\right)^{p} dz= o_p(1).
	\end{eqnarray}
	Finally for the last term in \eqref{GreenRepresPart} let us observe that  by the definition of $\varepsilon_{j,p}$ in \eqref{defEpsilon}
	\begin{equation}\label{logepsilon}
	\log\varepsilon_{j,p}=-\frac{(p-1)}{2}\log u_{p}(y_{j,p}) -\frac{1}{2}\log p,
	\end{equation}
	again by the dominated convergence theorem 
	\begin{equation}\label{aorasi}
	\lim_{p\rightarrow +\infty} \int_{B_{\frac{r}{\varepsilon_{j,p}}}(0)} \left(1+\frac{w_{j,p}(z)}{p}\right)^{p} dz = \int_{\R^2}e^{U}\overset{\eqref{LiouvilleEquationINTRO}}{=}  8\pi,
	\end{equation}
 and it follows
	\begin{eqnarray}\label{termineC_p}
	C_p &:=& -\frac{u_{p}(y_{j,p})\log\varepsilon_{j,p} }{2\pi p}\int_{B_{\frac{r}{\varepsilon_{j,p}}}(0)} \left(1+\frac{w_{j,p}(z)}{p}\right)^{p} dz
	\nonumber
	\\
	&\overset{\eqref{aorasi}}{=}& -\frac{u_{p}(y_{j,p})\log\varepsilon_{j,p} }{2\pi p} \ \left(8\pi +o_p(1)\right)
	\nonumber
	\\
	&\overset{\eqref{logepsilon}}{=}& u_{p}(y_{j,p}) \left[\frac{(p-1)}{p} \log u_{p}(y_{j,p}) +\frac{\log p}{p}  \right]\left(2+o_p(1)\right).
	\end{eqnarray}
	Substituting \eqref{termineA_p}, \eqref{termineB_p} and \eqref{termineC_p} into \eqref{GreenRepresPart} we get
	\[u_{p}(y_{j,p})=  u(y_{j,p}) \left[\frac{(p-1)}{p} \log u_{p}(y_{j,p}) +\frac{\log p}{p}  \right]\left(2+o_p(1)\right) +o_p(1),\]
	passing to the limit as $p\rightarrow +\infty$ and using \eqref{convValMax} conclude that
	\[\log m_j =\frac{1}{2}.\]
\end{proof}

\end{document}